\documentclass[11pt,a4paper]{amsart}
\usepackage[margin=1.03in]{geometry}
\usepackage{amsmath,amssymb,amsthm,amscd,mathtools,microtype,booktabs,array,enumitem}
\usepackage[svgnames]{xcolor}
\usepackage[colorlinks,linkcolor=FireBrick,citecolor=DarkGreen,urlcolor=MidnightBlue]{hyperref}
\hypersetup{
  pdftitle={Cusp restrictions and Bunke--Naumann invariants with level structure},
  pdfauthor={Yuqi Li and Hao-Yu Sun},
  pdfsubject={Topological modular forms with level structure and relative secondary invariants},
  pdfkeywords={TMF, Bunke--Naumann invariant, Gamma0(3), cusps, K-theory, modular forms}
}

\newcommand{\Z}{\mathbb Z}
\newcommand{\Q}{\mathbb Q}
\newcommand{\F}{\mathbb F}

\newcommand{\TMF}{\mathrm{TMF}}
\newcommand{\tmf}{\mathrm{tmf}}
\newcommand{\KO}{\mathrm{KO}}
\newcommand{\KU}{\mathrm{KU}}

\newcommand{\im}{\operatorname{im}}
\newcommand{\Exp}{\operatorname{Exp}}
\newcommand{\Tor}{\operatorname{Tor}}
\newcommand{\tr}{\operatorname{tr}}
\newcommand{\cK}{\mathcal K}
\newcommand{\cV}{\mathcal V}
\newcommand{\cO}{\mathcal O}

\newcommand{\Tate}{\operatorname{Tate}}

\theoremstyle{plain}
\newtheorem{theorem}{Theorem}[section]
\newtheorem{proposition}[theorem]{Proposition}
\newtheorem{lemma}[theorem]{Lemma}
\newtheorem{corollary}[theorem]{Corollary}
\theoremstyle{definition}
\newtheorem{definition}[theorem]{Definition}
\theoremstyle{remark}
\newtheorem{remark}[theorem]{Remark}

\title[Cusp restrictions and secondary invariants]{Cusp restrictions and Bunke--Naumann invariants\\with level structure}
\author{Yuqi Li}
\address{C. N. Yang Institute for Theoretical Physics, Stony Brook University, Stony Brook, New York 11794, USA}
\email{yuqi.li@stonybrook.edu}

\author{Hao-Yu Sun}
\address{Weinberg Institute, The University of Texas at Austin, 2515 Speedway, C1600, Austin, Texas 78712, USA}
\email{hkdavidsun@utexas.edu}
\date{10 September 2026}

\begin{document}

\begin{abstract}
Restriction to the full cusp divisor of a modular curve defines a
secondary invariant whose rational indeterminacy comes from a single
global modular form. For every integral weakly holomorphic level-one
modular form $h$ of weight divisible by four, we prove that the
imported value $[h/2]$ vanishes at every nontrivial $\Gamma_0(N)$
level. The construction localizes coefficients before completion
and rationalizes only after passing to homotopy groups. At odd prime
level, we identify the full cusp spectrum as a product of two real
Tate factors and construct a single holomorphic weight-two correction.
This correction yields an actual integral global homotopy class and
a rational global source class satisfying the equality required for
joint annihilation. The equality transports to every odd composite
level, while even levels follow by inverting two. At level three,
the Mahowald--Rezk homotopy calculation leaves only the periodic
$\nu$ family in stems $8K+3$. A cusp-residue homomorphism detects its secondary value, which has
exact order two throughout the periodic family. As an application,
the Bunke--Naumann secondary invariants of products in bidegrees
$(8k+1,8k'+2)$ vanish after passage to every nontrivial level.
\end{abstract}

\maketitle

\section{Introduction}

Restriction of topological modular forms to the Tate curve gives rise
to the Bunke--Naumann invariant, which records the ambiguity in a
choice of nullhomotopy through a modular-form indeterminacy
\cite[Sections~3.7 and~4.3]{BN}. Introducing level structure replaces
the single level-one cusp by the full cusp divisor of a modular curve.
The problem is then to correct the secondary value simultaneously at
all cusps using one global modular form.

Write $v_c$ for a representative at cusp $c$. Whereas separate cusp
quotients allow a different correcting form $f_c$ at each cusp, the
joint quotient requires the expansions of one rational global source
class, with an integral remainder in the homotopy of global cusp
sections. Integrality at every split complex chart does not by itself
establish this condition. In degree four, for example,
complexification sends the real generator $\alpha$ to twice the
complex Bott generator, so $\alpha/2$ has an integral complex image
without being an integral real class.

Our proof retains this integral lattice in real K-theory. For an odd
prime $p$, the full punctured cusp spectrum is
\[
 R_{\Z[1/p]}(q)\times R_{\Z[1/p]}(t),
 \qquad q\longmapsto(q,t^p),
\]
where $R_S$ denotes the completed real K-theory spectrum of a
Tate-curve chart over $S$. A single holomorphic weight-two form $F_p$
has expansions congruent to $1$ modulo $2$ in both factors.
Multiplication by the given level-one form $h$ supplies a rational
global source class, while the two series
$(1-\Exp_c(F_p))h_c/2$ define an actual integral global class.
The equality involving these classes proves joint annihilation.
At composite level, both witnesses are transported through the
level-forgetting map, without requiring an additional
sheaf-to-homotopy lifting assertion.

The order of completion is part of the construction. At level $N$,
we use the coefficient ring $S_N=\Z[1/N]$ and the rational Laurent
numerator $S_N((q))\otimes\Q$, whose coefficients have a common
denominator away from $N$. This numerator is generally smaller than
$\Q((q))$. Although the inherited target $\KO((q))[1/N]$ maps to
the completed target used here, we do not identify the two targets
as a whole. Section~\ref{sec:cusps} specifies the maps, lattices,
and normalization needed for the subsequent construction of the
witnesses.

The relative-invariant construction and its two indeterminacies follow
Bunke--Naumann \cite[Section~4.3]{BN}, whose coefficient convention
already distinguishes rationalized integral series from arbitrary
rational series \cite[Section~4.1]{BN}. Hill--Lawson construct the
level spectra and cusp restrictions \cite[Sections~5--6]{HL}, while
Mazur's congruence and normalized involution supply the prime-level
arithmetic \cite[II, Sections~5--6]{Mazur}. Within this framework,
the point requiring a separate argument is the simultaneous equality
in actual real global homotopy, with coefficient localization preceding
completion. We establish that equality by constructing both the
integral global class and the rational global source class.

The following theorem combines joint annihilation with the published
level-three torsion calculation, a residue detection argument, and
the product application. Write $\delta\in\pi_{48}\TMF_0(3)$ for
the Mahowald--Rezk periodicity unit whose modular form is
$\Delta^2$ \cite[Proposition~4.1]{MR}, and write $\nu$ for the image
of the stable degree-three Hopf class. The distinction between the
spectral unit $\delta$ and the modular form $\Delta^2$ will also
keep track of the negative periods.

\begin{theorem}[Main theorem]\label{thm:main}
Let
\[
 \Phi_N^\partial:\TMF_0(N)\longrightarrow\cK_N^\partial
\]
be the full-cusp map of Proposition~\ref{prop:full-cusp-map}, and let $b_N^\partial$ be its relative
secondary invariant.
\begin{enumerate}[label=\textup{(\roman*)},leftmargin=2.8em]
\item For every $N>1$, every $K\in\Z$, and every
$h\in M_{4K}^{!}(\mathrm{SL}_2(\Z);\Z)$, the image of $[h/2]$ in the joint all-cusp value group
is zero.
\item For every $m\in\Z$,
\[
 \Tor\pi_{48m+3}\TMF_0(3)=\Z/2\{\nu\delta^m\},
\]
and the torsion in stems $48m+11,48m+19,48m+27,48m+35,48m+43$ is zero.
\item The class $\nu\delta^m$ lies in $\ker(\Phi_3^\partial)_*$.
\[
 b_3^\partial(\nu\delta^m)
 =\rho_{3,6m}^\partial\!\left(\left[\frac{E_2\Delta^{2m}}{24}\right]\right)
\]
is nonzero of exact order two.
\item In every stem $8K+3$, the kernel of $b_3^\partial$ on its full torsion domain is
zero, and its image is $\Z/2$ precisely when $K\equiv0\pmod6$.
\item For all integers $k,k'$ and classes
$a\in\pi_{8k+1}\TMF$, $a'\in\pi_{8k'+2}\TMF$,
\[
 b_N^\partial\bigl(\iota_N(aa')\bigr)=0\qquad(N>1).
\]
\end{enumerate}
\end{theorem}

Part~\textup{(i)} is proved in Sections~\ref{sec:cusps}
and~\ref{sec:joint}. Part~\textup{(ii)} and the primary kernel in
part~\textup{(iii)} follow by degree bookkeeping from the
Mahowald--Rezk calculation. The residue relation in
Section~\ref{sec:residue} annihilates every rational source
indeterminacy but detects the value of $\nu$; the actual periodic
unit then proves parts~\textup{(iii)--(iv)} in all degrees.
Section~\ref{sec:products} proves the product formula in the bounded
target, following Tachikawa's argument, and combines the published
primary-image theorem with part~\textup{(i)} to obtain part~\textup{(v)}.

\section{Relative secondary classes}

A class killed by a map $E\to F$ admits a lift to its cofiber. When
the class is torsion, the rationalization of this lift comes from a
class of $F$. The secondary invariant records the resulting rational
representative modulo two choices: an integral class of $F$ and a
rational class coming from $E$. We use this construction at each level.

\begin{proposition}[Relative cofiber class]\label{prop:relative}
Let $\varphi:E\to F$ be a map of spectra, and let $x\in\pi_nE$ be torsion with
$\varphi_*(x)=0$. There is a natural additive class
\begin{equation}\label{eq:relative}
 b_\varphi(x)\in
 \frac{\pi_{n+1}F\otimes\Q}
 {\im(\pi_{n+1}F)+\varphi_*(\pi_{n+1}E\otimes\Q)}.
\end{equation}
\end{proposition}

\begin{proof}
Let $C_\varphi$ be the cofiber, and choose
$\widetilde x\in\pi_{n+1}C_\varphi$ with boundary $x$.
Since $x$ is torsion, rational exactness supplies
$z\in\pi_{n+1}F\otimes\Q$ mapping to $\widetilde x_\Q$.
Changing $z$ adds an element of
$\varphi_*(\pi_{n+1}E\otimes\Q)$, whereas changing the integral lift
adds an element of the integral image of $\pi_{n+1}F$.
Compatible choices under addition establish additivity.
\end{proof}

\begin{proposition}[Naturality]\label{prop:naturality}
A homotopy-commutative square
\[
\begin{CD}
 E @>{\varphi}>> F\\
 @V{u}VV @VV{v}V\\
 E' @>{\varphi'}>> F'
\end{CD}
\]
induces a map of the quotients in~\eqref{eq:relative}, and
\[
 b_{\varphi'}(u_*x)=v_*b_\varphi(x).
\]
\end{proposition}

\begin{proof}
The square induces a map of cofiber sequences. Transporting the
integral cofiber lift and its rational representative through this
map gives the stated identity.
\end{proof}

\begin{lemma}[Multiplication by a unit]\label{lem:unit}
Let $\varphi:E\to F$ be a map of associative ring spectra, and let
$u\in\pi_dE$ be a unit. With ordinary left multiplication, for every
$x$ in the domain of $b_\varphi$,
\[
 b_\varphi(ux)=(-1)^d\varphi(u)b_\varphi(x).
\]
In particular, the formula has no sign when $d$ is even.
\end{lemma}

\begin{proof}
The cofiber sequence is a sequence of left $E$-modules, with $E$
acting on $F$ through $\varphi$. Because its boundary has degree
$-1$,
\[
 \partial(u\widetilde x)=(-1)^d u\partial(\widetilde x).
\]
The sign arises from interchanging the degree-$d$ sphere with the
suspension coordinate in $\Sigma E$. Thus, if $\widetilde x$ lifts
$x$ and $z$ is its rational representative, then
$(-1)^d u\widetilde x$ lifts $ux$ and has rational representative
$(-1)^d\varphi(u)z$. Multiplication by $\varphi(u)$ carries each
indeterminacy subgroup to the corresponding subgroup in the shifted
degree, while multiplication by $\varphi(u^{-1})$ provides the
inverse. Passing to the quotients proves the formula.
\end{proof}

\section{Completed cusp spectra and their integral lattices}\label{sec:cusps}

The integral lattice in a secondary target is determined by the homotopy
of a spectrum. Computing this lattice requires retaining both the order
of completion and the real descent at a cusp. We specify these operations
before computing the full target at prime level, where it has two real
factors.

Let $S_N=\Z[1/N]$. We use the normalized Hill--Lawson compactification
$\overline{\mathcal M}_0(N)$ over $S_N$, with smooth locus
$\mathcal M_0(N)$ and full cusp divisor $D_N$. Write
\[
 E_N=\TMF_0(N)=\Gamma(\mathcal M_0(N),\cO^{\mathrm{top}}).
\]
Here $E_N$ denotes a spectrum; the notation $\mathcal E_p$ below
denotes a modular form.

\subsection{Completion before the rational quotient}
For $S=\Z[1/N]$, define the completed real Tate spectrum by
\begin{equation}\label{eq:ordered-completion}
 R_S(t)=\left(\operatorname*{holim}_r
        (\KO[1/N])[t]/t^r\right)[t^{-1}].
\end{equation}
The truncation is the finite truncated monoid algebra, whose coefficient
module has the $r$ summands $1,t,\ldots,t^{r-1}$. Thus coefficient
localization precedes completion in the parameter, which in turn
precedes inversion of the parameter. Homotopy groups are rationalized
only when forming the quotient~\eqref{eq:relative}.

\begin{lemma}[The completed homotopy groups]\label{lem:completed-coefficients}
In every degree $j$,
\[
 \pi_jR_S(t)=(\pi_j\KO[1/N])((t)).
\]
In particular, with $e_K=\alpha\beta^K$, $|\alpha|=4$, $|\beta|=8$,
and $\beta$ invertible,
\begin{equation}\label{eq:real-lattice}
 \pi_{8K+4}R_S(t)=S((t))e_K,\qquad K\in\Z.
\end{equation}
\end{lemma}
\begin{proof}
In degree $j$, the truncations have homotopy
$\bigoplus_{a=0}^{r-1}(\pi_j\KO[1/N])t^a$. Because their transition maps
are surjective in every degree, the Milnor $\lim^1$ term vanishes.
The inverse limit therefore has the product of the coefficient groups,
one for each nonnegative power of $t$. The subsequent localization,
being a filtered telescope, has homotopy consisting of series with a
finite lower Laurent bound. This proves the assertion without a
connectivity assumption on periodic $\KO$.
\end{proof}

To make the rationalization convention explicit, put
\begin{equation}\label{eq:bounded-numerator}
 \Q_S((t))_{\mathrm{bd}}:=S((t))\otimes_{\Z}\Q
 =\{a\in\Q((t)): Da\in S((t))\text{ for some }D\in\Z\setminus\{0\}\}.
\end{equation}
A single denominator $D$ must clear all coefficients away from the primes
dividing $N$. For a prime $\ell\nmid N$, the series
$\sum_{j\ge0}t^j/\ell^j$ lies in $\Q((t))$ but not in
$\Q_S((t))_{\mathrm{bd}}$. Likewise, for $N>1$, $\sum_{j\ge0}t^j/N^j$ belongs to
$S((t))$ but not to $\Z((t))[1/N]$. The rationality of each
individual coefficient does not remove either distinction.

There is a natural comparison
\begin{equation}\label{eq:old-new-comparison}
 \KO((q))[1/N]\longrightarrow R_{S_N}(q).
\end{equation}
The coefficient maps on the finite truncations induce a map first on
their inverse limits and then on their Laurent localizations. This map
factors through $[1/N]$, since $N$ is invertible in the target.
We use the forward comparison for inherited level-one classes; no
inverse or equivalence of the two completed targets is asserted.

\subsection{The full punctured neighborhood}
At the standard cusp, the Tate curve carries the subgroup $\mu_N$.
With the completion convention just fixed, its real restriction map is
\begin{equation}\label{eq:standard-map}
 \Phi_N^\infty:E_N\longrightarrow R_{S_N}(q).
\end{equation}

\begin{proposition}[The punctured full-cusp map]\label{prop:full-cusp-map}
There is a natural map of commutative ring spectra
\begin{equation}\label{eq:full-cusp-map}
 \Phi_N^\partial:E_N\longrightarrow\cK_N^\partial
 :=\Gamma\!\left(\widehat{\overline{\mathcal M}_0(N)}_{D_N}^{\,\times},
                    \cO^{\mathrm{top}}\right).
\end{equation}
Here punctured completion means the completed Tate charts with their
cusp parameter inverted, followed by spectral descent. The construction
includes all cusp components and their descent data and is natural
under forgetting level structure.
\end{proposition}
\begin{proof}
Hill--Lawson realize the completed Tate charts by inverse limits of
finite monoid truncations over multiplicative $K$-theory
\cite[Section~5.1, Theorem~5.7 and Corollary~5.8]{HL}. On each chart,
we first localize the coefficients over $S_N$, then complete its
parameter, and finally invert that parameter. The resulting maps
commute with the chart morphisms. The smooth/Tate comparison and
arithmetic gluing
\cite[Propositions~5.9 and~5.12, Section~5.3]{HL} supply restriction
from the smooth locus, and taking sections of the descent diagram gives
\eqref{eq:full-cusp-map}. This construction is distinct from the empty
ordinary open formal subscheme obtained by inverting an ideal of
definition.

A level-forgetting morphism extends to the normalized compactification
\cite[Proposition~3.19 and Theorem~6.1]{HL}. Locally, a cusp parameter
pulls back to a unit times a positive power of a cusp parameter, so the
resulting completion filtrations are cofinal. The pulled-back parameter
becomes invertible after puncturing, and the compatible Tate
realizations therefore give the claimed naturality on actual spectra.
\end{proof}

\begin{definition}[Joint all-cusp invariant]\label{def:full-cusp}
For a torsion class $y\in\pi_{8K+3}E_N$ killed by
$(\Phi_N^\partial)_*$, let $b_N^\partial(y)=b_{\Phi_N^\partial}(y)$,
with target
\begin{equation}\label{eq:joint-value}
 \cV_{N,K}^\partial=
 \frac{\pi_{8K+4}\cK_N^\partial\otimes\Q}
 {\im\pi_{8K+4}\cK_N^\partial+
  (\Phi_N^\partial)_*(\pi_{8K+4}E_N\otimes\Q)}.
\end{equation}
\end{definition}
The numerator and both indeterminacies in this expression are homotopy
groups of global sections. In particular, the definition of a section
of a homotopy sheaf does not make it an integral element of this
denominator.

\subsection{The two real factors at an odd prime}
\begin{lemma}[The full prime-level cusp spectrum]\label{lem:prime-cusps}
Let $p$ be an odd prime and $S=\Z[1/p]$. Using the invariant elliptic
differential $d\log(z)$ on the Tate curve, there is an equivalence
\begin{equation}\label{eq:prime-cusps}
 \cK_p^\partial\simeq R_S(q)\times R_S(t).
\end{equation}
The two maps from the completed level-one cusp have parameters
$q\mapsto q$ and $q\mapsto t^p$, respectively. Their elliptic curves
and subgroups are
\begin{equation}\label{eq:tate-branches}
 (\Tate(q),\mu_p),\qquad (\Tate(t^p),\langle t\rangle).
\end{equation}
Thus~\eqref{eq:prime-cusps} identifies actual integral homotopy groups,
including the real lattice at the prime two.
\end{lemma}
\begin{proof}
Put $A=S((q))$. On the punctured Tate chart, the finite etale group
$\Tate(q)[p]$ is an extension of $\Z/p$ by $\mu_p$. Locally, a cyclic
subgroup of order $p$ either equals $\mu_p$ or projects isomorphically
onto $\Z/p$. In the latter case, its unique element over $1$ is
represented by a root $t$ of $t^p=q$. This fppf-local description
represents the full finite etale subgroup scheme, with algebra
\begin{equation}\label{eq:subgroup-algebra}
 B_p=A\times A[t]/(t^p-q)=S((q))\times S((t)).
\end{equation}
The polynomial extension is etale because $pt^{p-1}$ is a unit.
For the second equality, group a Laurent series in $t$ by its exponent
modulo $p$: each of the $p$ resulting series in $q=t^p$ has a finite
lower bound. The algebra has rank $p+1$, so neither cusp branch has
been omitted.

Before puncturing, the normalized branches are $S[[q]]$ and $S[[t]]$.
They are finite and normal over $S[[q]]$, with the fraction algebras
in~\eqref{eq:subgroup-algebra}, and therefore compute the normalization.
Their cusp widths are $1$ and $p$, as in the level compactification of
\cite[Proposition~3.19]{HL}.

The completed level-one Tate stack retains the constant group $C_2$,
which acts by inversion on the elliptic curve
\cite[Proposition~3.15]{HL}. Inversion fixes each cyclic subgroup as a
subgroup and hence fixes $B_p$, although it would not fix a chosen
generator for $\Gamma_1(p)$-structure. This inversion action persists
in characteristic two: in the coordinate $x=z-1$, its formal action
is $x\mapsto-x/(1+x)$, which over $\F_2$ is
$x+x^2+x^3+\cdots$, not $x$.

We next compute spectral sections of these stack components. Set
$R=R_S(q)$ and $U=(\KU[1/p])[[q]][q^{-1}]$, with the same ordered
completion as in~\eqref{eq:ordered-completion}. The etale $R$-algebra
$R_{B_p}$ associated to~\eqref{eq:subgroup-algebra} has
\[
 \pi_*R_{B_p}=\pi_*R\otimes_A B_p.
\]
Etale algebras, including their morphisms, are classified by their
degree-zero etale algebra even for nonconnective ring spectra
\cite[Theorem~2.12, Corollary~2.13 and Example~2.14]{MM}.
By the same grouping of powers, the explicit algebra
$R_S(q)\times R_S(t)$ has these homotopy groups in every degree
and thus realizes $R_{B_p}$.

On the complex Tate chart, the level pullback is the corresponding
etale $U$-algebra, whose inversion action is conjugation on $\KU$
with the subgroup coefficients fixed. The Tate realization and
smooth/Tate comparison cited above identify this full coherent descent
datum with the complex base change of $R_{B_p}$.

The compatibility of this descent with the specified completion uses
the faithful $C_2$-Galois extension $\KO\to\KU$ and the dualizability
of $\KU$ over $\KO$ supplied by Wood's theorem
\cite[Section~5.1, Example~5.6, Theorem~5.7 and Example~6.1]{MM}.
Consequently, tensoring with $\KU$ commutes with the inverse limit
of the truncations and with Laurent localization, giving
\[
 R\otimes_{\KO}\KU\simeq U.
\]
Faithful Galois descent is preserved under base change. Hence
\[
 R\simeq U^{hC_2},\qquad
 R_{B_p}\simeq(U\otimes_RR_{B_p})^{hC_2}.
\]
Taking sections of the two actual Tate components now proves
\eqref{eq:prime-cusps}. The argument neither interchanges arbitrary
homotopy fixed points with a Laurent colimit nor replaces homotopy
fixed points by invariants of homotopy groups.
\end{proof}

Complexification sends the compatible Bott generators to
\begin{equation}\label{eq:complexification}
 c(\alpha)=2u^2,\qquad c(\beta)=u^4,\qquad
 c(e_K)=2u^{4K+2},\quad |u|=2.
\end{equation}
Thus $e_K/2$ is not integral in either real factor of
\eqref{eq:prime-cusps}, although its complexification is integral.
This distinction makes the real spectrum computation essential to
the witness constructed below.

\subsection{A rational modular form in global homotopy}
\begin{lemma}[Rational global source]\label{lem:rational-source}
For every integer $w$, naturally under level restriction,
\begin{equation}\label{eq:rational-source}
 \pi_{2w}E_N\otimes\Q
 \cong M_w^!(\Gamma_0(N);\Q).
\end{equation}
The edge map is Tate expansion with the invariant elliptic differential.
For $w=4K+2$, its coordinate in a real cusp factor is
\begin{equation}\label{eq:half-normalization}
 \frac12\Exp_c(f)e_K.
\end{equation}
\end{lemma}
\begin{proof}
The stack $X=\mathcal M_0(N)$ is noetherian and separated. After $N$
is inverted, its level-forgetting map is representable finite etale,
so its map to the formal-group stack is flat. Its relative inertia
kernel is trivial: an elliptic-curve automorphism identical on the
formal completion at the identity is identical on the curve.
For example, equality on the completed local ring forces equality
on the function field and hence on the smooth proper curve. This
verifies relative tameness even in characteristics where the full
modular stabilizer has noninvertible order.

The global-sections theorem of Mathew--Meier
\cite[Definition~2.28 and Theorem~4.14]{MM} therefore applies to the
even-periodic refinement and shows that global sections commute with
homotopy colimits. In particular, it justifies the comparison
\[
 E_N\otimes\Q\simeq\Gamma(X,\cO^{\mathrm{top}}\otimes\Q).
\]
With this comparison established, we calculate rational descent.

Over $\Q$, choose a full auxiliary level $m\ge3$ divisible by $N$,
retaining all its components. The resulting smooth modular scheme
$Y$ is affine, since a nonempty cusp divisor has been removed from
each proper modular-curve component. The stack $X_{\Q}$ is $[Y/G]$
for a finite group $G$. Because taking invariants is exact over $\Q$,
$H^s(X_{\Q},\omega^w)=0$ for $s>0$ and every integer $w$.
The rational descent spectral sequence therefore has cohomological
dimension zero and converges to the edge
identification~\eqref{eq:rational-source}, including negative weights.
Sections on the smooth stack allow finite poles at the omitted cusps.

The edge map is natural under Tate restriction. In the complex
coordinate, the expansion is $\Exp_c(f)u^w$;
equation~\eqref{eq:complexification} then gives
\eqref{eq:half-normalization}. The source is an actual rational global
homotopy class, not only a rational homotopy-sheaf section.
\end{proof}

It follows that the standard-cusp value group, in the coordinate $e_K$, is
\begin{equation}\label{eq:standard-value}
 \cV_{N,K}^{\infty}\cong
 \frac{\Q_{S_N}((q))_{\mathrm{bd}}}
 {S_N((q))+
  \{\tfrac12\Exp_\infty(f):f\in M_{4K+2}^!(\Gamma_0(N);\Q)\}}.
\end{equation}
The source image lies in the bounded numerator because it is the
image of $\pi_{8K+4}E_N\otimes\Q$ under an actual spectrum map.
Likewise, the rational source in the joint quotient is obtained by
simultaneously restricting one global class, rather than choosing
a class independently at each cusp.

Let $\cV_{1,K}$ denote the quotient~\eqref{eq:relative} for the
inherited level-one map $\Phi_1:\TMF\to\KO((q))$. Its localization
$\cV_{1,K}[1/N]$ uses the localized inherited target, before the
forward comparison~\eqref{eq:old-new-comparison}.
\begin{corollary}[Naturality from level one]\label{cor:level-naturality}
The forgetful map $\iota_N:\TMF\to E_N$ and cusp restriction induce
\[
 \rho_{N,K}^\partial:\cV_{1,K}[1/N]\longrightarrow\cV_{N,K}^\partial,
 \qquad
 b_N^\partial(\iota_Nx)=\rho_{N,K}^\partial(b_1(x))
\]
for every torsion $x\in\pi_{8K+3}\TMF[1/N]$.
\end{corollary}
\begin{proof}
Apply Proposition~\ref{prop:naturality} to the forward comparison
and the actual level-forgetting restriction square of
Proposition~\ref{prop:full-cusp-map}. Since
$\pi_{8K+3}\KO((q))=0$, the level-one primary image is zero;
the imported class therefore lies in the primary kernel at level $N$.
\end{proof}

\section{One arithmetic form and two global homotopy witnesses}\label{sec:joint}

The prime-level target has now been identified as an actual product
of real spectra. We construct one weight-two form whose two expansions
are congruent to $1$ modulo $2$ and use these expansions to write an
integral element of the target. A separate rational global class
supplies the correction. Both classes are then transported to composite
level through a spectrum map.

An elementary coefficient test clarifies the distinction between this
construction and a sheaf-level integrality argument.
\begin{lemma}[Faithfully flat integrality]\label{lem:ff-integral}
Let $A\to B$ be faithfully flat and let $M\subset M_\Q$ be an $A$-module inside its
rationalization.  If $x\in M_\Q$ maps into $M\otimes_A B$, then $x\in M$.  If moreover
$x\otimes1$ is divisible by $2$ in $M\otimes_A B$, then $x$ is divisible by $2$ in $M$.
\end{lemma}

\begin{proof}
The class of $x$ in $(M_\Q/M)$ vanishes after tensoring with the
faithfully flat algebra $B$, and hence vanishes already. For the second
statement, apply faithful flatness to the class of $x$ in $M/2M$.
\end{proof}
The lemma tests membership in a specified module; when applied to a
homotopy sheaf, it does not produce a lift to the homotopy of its global
sections. Lemma~\ref{lem:prime-cusps} supplies the actual module in
which we construct the integral witness below.

\subsection{The second-cusp expansion, including its constant term}
Use the determinant-normalized weight-two slash operator
\[
 (f|\gamma)(z)=\det(\gamma)(cz+d)^{-2}f(\gamma z),
 \qquad
 W_p=\begin{pmatrix}0&-1\\p&0\end{pmatrix},\quad
 S_0=\begin{pmatrix}0&-1\\1&0\end{pmatrix},\quad
 T=\begin{pmatrix}1&1\\0&1\end{pmatrix}.
\]
\begin{lemma}[Prime-level weight-two trace]\label{lem:prime-trace}
If $f\in M_2(\Gamma_0(p);\Q)$ is holomorphic at all cusps, with
standard expansion $\sum_{r\ge0}a_rq^r$, its expansion on the second
chart of~\eqref{eq:tate-branches} is
\begin{equation}\label{eq:second-expansion}
 \Exp_0(f)(t)=-\frac1p\sum_{r\ge0}a_{pr}t^r.
\end{equation}
\end{lemma}
\begin{proof}
For $t=\exp(2\pi iz)$, the chart
$(\Tate(t^p),\langle t\rangle,d\log)$ gives
\[
 \Exp_0(f)(t)=(f|S_0)(pz)=p^{-1}(f|W_p)(z).
\]
Indeed, multiplication by $pz$ identifies the complex elliptic curve
at $-1/(pz)$, equipped with its standard order-$p$ subgroup, with the
curve at $pz$ and subgroup generated by $z$. Under this identification,
transporting the elliptic differential gives $(pz)^{-2}f(-1/(pz))$,
which accounts for the displayed factor $p^{-1}$ in the
determinant-normalized $W_p$.

The left cosets of $\Gamma_0(p)$ in $\mathrm{SL}_2(\Z)$ have
representatives $I,S_0T^j$, $0\le j<p$. Since
$W_pS_0T^j=-\left(\begin{smallmatrix}1&j\\0&p\end{smallmatrix}\right)$,
the trace of $f|W_p$ is
\[
 f|W_p+\frac1p\sum_{j=0}^{p-1}f((z+j)/p)=f|W_p+U_pf,
 \qquad U_pf=\sum_{r\ge0}a_{pr}q^r.
\]
The root-of-unity average selects precisely the coefficients indexed
by multiples of $p$, including the constant coefficient. The trace,
being a holomorphic weight-two form at level one, therefore vanishes:
its differential on $X(1)=\mathbb P^1$ has at most a simple pole at
the sole cusp. The residue theorem removes this possible pole, and
no regular differentials remain. Invariant differentials descend regularly at
elliptic points, which therefore contribute no additional poles.
Thus $f|W_p=-U_pf$, proving~\eqref{eq:second-expansion}.
\end{proof}
These hypotheses are essential: we apply the trace to a holomorphic
weight-two form before multiplying by the possibly weakly holomorphic
form $h$.

\subsection{One simultaneous congruence}
Normalize $E_2=1-24\sum_{r\ge1}\sigma_1(r)q^r$. For a prime $p$, put
\[
 \delta_p(q)=\sum_{r\ge1}\sigma'_p(r)q^r,
 \qquad \sigma'_p(r)=\sum_{\substack{a\mid r\\p\nmid a}}a.
\]
\begin{lemma}[Mazur's half-constant witness]\label{lem:Mazur}
For every odd prime $p$ there is a holomorphic
$G_p\in M_2(\Gamma_0(p);\Q)$ with
\[
 \Exp_\infty(G_p)=\frac12+A_p(q),\qquad A_p(q)\in\Z[[q]].
\]
Writing $F_p=2G_p$, the same form satisfies
\begin{equation}\label{eq:prime-congruences}
 \Exp_\infty(F_p)\in1+2\Z[[q]],\qquad
 \Exp_0(F_p)\in1+2\Z[1/p][[t]].
\end{equation}
\end{lemma}
\begin{proof}
First take $p\ge5$, and set
$d_p=\gcd(p-1,12)$ and $n_p=(p-1)/d_p$.
Mazur's congruence states that $\delta_p$ modulo $n_p$ is parabolic
\cite[II, Proposition~(5.12)(iii), pp.~85--87]{Mazur}. Here parabolic
has a specific lattice meaning: his $B^0(\Z/n_p)$ is
$B^0(\Z)\otimes\Z/n_p$, the quotient of the characteristic-zero
integral cusp-form lattice
\cite[II, Section~4(1), pp.~69--70, and (4.9), p.~77]{Mazur}.
Consequently, there is an actual integral rational cusp form $g_p$
satisfying
\[
 \delta_p-g_p=n_pb_p,\qquad b_p\in q\Z[[q]].
\]
This step uses the quotient of that lattice and makes no lifting
claim for arbitrary characteristic-two modular forms. If $n_p=1$,
take $g_p=0$.

Define
\begin{equation}\label{eq:prime-form}
 \mathcal E_p=pE_2(p\tau)-E_2(\tau),\qquad
 F_p=\frac{\mathcal E_p-24g_p}{p-1},\qquad G_p=\frac12F_p.
\end{equation}
The transformation anomaly cancels in $\mathcal E_p$, and
$\mathcal E_p|W_p=-\mathcal E_p$, so this form is holomorphic at
both cusps. Moreover, $\mathcal E_p=(p-1)+24\delta_p$, which gives
\begin{align}
 f_\infty(q):=\Exp_\infty(F_p)
   &=1+\frac{24}{d_p}b_p(q),\label{eq:prime-infinity}\\
 f_0(t):=\Exp_0(F_p)
   &=-\frac1p-\frac{24}{pd_p}(U_pb_p)(t).
   \label{eq:prime-zero}
\end{align}
The second equality follows from Lemma~\ref{lem:prime-trace}.
Since $d_p\mid12$, the nonconstant coefficients are even over
$\Z[1/p]$, and $-1/p-1=-(p+1)/p$ is also even there. Although
the second constant is $-1/p$ rather than $1$, its denominator
is allowed in the prescribed lattice.

For $p=3$, which lies outside the range of the cited Mazur result,
take directly
\begin{equation}\label{eq:level-three-form}
 F_3=\frac{3E_2(3\tau)-E_2(\tau)}2=1+12\delta_3(q).
\end{equation}
Since $\sigma'_3(3r)=\sigma'_3(r)$, Lemma~\ref{lem:prime-trace} gives
$\Exp_0(F_3)=-1/3-4\delta_3(t)$. This proves both congruences
in~\eqref{eq:prime-congruences} and completes the construction.
\end{proof}

\begin{proposition}[Prime-level global parity]\label{prop:global-parity}
Let $p$ be odd and let $G_p$ be the form just constructed. At every
geometric cusp $c$ of $X_0(p)$, the expansion
\[
 \Exp_c(G_p)-\frac12
\]
is integral over $\Z[1/p]$ in the completed cusp ring with its Tate
differential convention.
\end{proposition}
\begin{proof}
Lemma~\ref{lem:prime-cusps} identifies the two branches of the full
prime cusp neighborhood, and~\eqref{eq:prime-congruences} proves the
assertion on each. The asserted integrality persists under extension
of their coefficient rings to geometric cusp charts.
\end{proof}

\begin{theorem}[Simultaneous half-constant witness]\label{thm:simultaneous}
For every odd $N>1$, there is one holomorphic
$G_N\in M_2(\Gamma_0(N);\Q)$ such that
$\Exp_c(G_N)-1/2$ is integral at every geometric cusp $c$ of $X_0(N)$.
\end{theorem}
\begin{proof}
Choose an odd prime $p\mid N$ and pull $G_p$ back by
$(E,C_N)\mapsto(E,C_N[p])$. This map leaves the elliptic curve and
its Tate differential unchanged. On completed cusp charts, the
parameter pulls back to a unit times a positive power of the finer
parameter. Substitution therefore preserves both integrality and the
congruence to $1/2$, after extending coefficients from $\Z[1/p]$
to $\Z[1/N]$.
\end{proof}

\subsection{The equality in global homotopy}
\begin{theorem}[Joint annihilation of half-integral level-one classes]\label{thm:joint-kill}
For every $N>1$, every $K\in\Z$, and every
$h\in M_{4K}^!(\mathrm{SL}_2(\Z);\Z)$,
\[
 \rho_{N,K}^\partial\!\left(\left[\frac h2\right]\right)=0
 \quad\text{in }\cV_{N,K}^\partial.
\]
\end{theorem}
\begin{proof}
First suppose $N=p$ is an odd prime, put $S=\Z[1/p]$, and use the
two factors of Lemma~\ref{lem:prime-cusps}. The imported integral
series has expansions $h_\infty=h(q)$ and $h_0=h(t^p)$. Define
\begin{equation}\label{eq:integral-witness}
 z_p=\left(
    \frac{1-f_\infty(q)}2h(q)e_K,
    \frac{1-f_0(t)}2h(t^p)e_K
       \right)\in\pi_{8K+4}\cK_p^\partial.
\end{equation}
This is a class in actual homotopy: the two modules are
$S((q))e_K$ and $S((t))e_K$, and the congruences
\eqref{eq:prime-congruences} make both coefficients integral.
The expansion $h(t^p)$ is forced by the underlying curve $\Tate(t^p)$. Forgetting a subgroup of that same elliptic curve
introduces no isogeny or Adams-operation factor.

The single rational modular form $m_p=F_ph$ has weight $4K+2$ and
is holomorphic on the smooth level stack. Lemma~\ref{lem:rational-source}
supplies its rational global class
$x_p\in\pi_{8K+4}E_p\otimes\Q$, whose two coefficients are
$(f_\infty h_\infty/2)e_K$ and $(f_0h_0/2)e_K$ by
\eqref{eq:half-normalization}. Thus, writing $j_p$ for the
map from the inherited level-one cusp spectrum,
\begin{equation}\label{eq:global-witness-equality}
 j_{p*}\bigl((h/2)e_K\bigr)
       -(\Phi_p^\partial)_*(x_p)=z_p\otimes1
 \quad\text{in }\pi_{8K+4}\cK_p^\partial\otimes\Q.
\end{equation}
All terms lie in the bounded rationalizations of the two integral
Laurent modules. The injections of these rationalizations into the
rational coefficient series show that the displayed coefficient
identity is an equality in that numerator. Because the integral class
was constructed before this comparison, no integral descent obstruction
has been discarded.

Now let $N>1$ be any odd integer and choose $p\mid N$. The canonical
subgroup $C_N[p]$ gives the restriction square of actual spectra
\begin{equation}\label{eq:prime-to-composite}
 \begin{CD}
 E_p @>{\Phi_p^\partial}>>\cK_p^\partial\\
 @VVV @VVV\\
 E_N @>{\Phi_N^\partial}>>\cK_N^\partial.
 \end{CD}
\end{equation}
The square uses the same level-forgetting map on the smooth stack
and on its punctured completed charts, by
Proposition~\ref{prop:full-cusp-map}. In the lower row, additional
coefficient primes are inverted before completion. Transport $x_p$
and $z_p$ through this square. Equation~\eqref{eq:global-witness-equality}
then holds with their images $x_N$ and $z_N$, expressing the imported
half as an integral global class plus the image of one rational global
source class. Its class in~\eqref{eq:joint-value} is therefore zero.
Prime powers and general odd composites require no new integral-lifting
theorem.

Finally, if $N$ is even, the image of $h(q)e_K$ is integral in
$\pi_{8K+4}\cK_N^\partial$, where multiplication by $2$ is invertible.
Half that image is therefore already an integral global class; take
it as $z_N$ and take $x_N=0$. This proves the remaining case.

For negative $K$, the Bott class $\beta^K$ still exists and every
expansion has a finite negative Laurent tail. The zero form has zero
witnesses. The excluded level $N=1$ is not covered by either argument.
\end{proof}

For example, at $p=3$, $K=0$, $h=1$, the integral witness is
\[
 z_3=\bigl(-6\delta_3(q),\;2/3+2\delta_3(t)\bigr)e_0.
\]
Its constant coefficients are $0$ and $2/3$, both integral over
$\Z[1/3]$. Together with the rational class of $F_3$, it gives
\eqref{eq:global-witness-equality} in both real factors.

\section{Bunke--Naumann secondary invariants of products}\label{sec:products}

We apply the joint-annihilation theorem to ordinary multiplication in
$\TMF$, followed by the relative invariant $b_1$. Tachikawa's product
formula evaluates this invariant using the two primary Tate images
\cite[Proposition~1.1 and Section~1.3]{Tachikawa}.
We give its proof in the bounded rational target used here. The
$\eta$-product case also appears in
\cite[Section~3.1, equation~(3.5)]{TYYmod2}.

Put $L=\Z((q))$ and $L_\Q=L\otimes\Q$. In the real coordinate
$e_K=\alpha\beta^K$, the level-one target is
\begin{equation}\label{eq:bounded-product-target}
 \cV_{1,K}=\frac{L_\Q}
 {L+M_{4K+2}^{!}(\mathrm{SL}_2(\Z);\Q)}.
\end{equation}
Here a modular form denotes its Tate expansion. Such expansions lie
in $L_\Q$: a rational weakly holomorphic level-one form is a finite
rational polynomial in $c_4,c_6,\Delta^{-1}$, whose expansions are
integral. The factor $1/2$ in~\eqref{eq:half-normalization} leaves
the rational source subspace unchanged. We retain $L_\Q$ throughout;
there is no identification with unrestricted $\Q((q))$.

\begin{lemma}[The real half-value]\label{lem:KO-half}
Let $C_{\KO}=\operatorname{cofib}(\KO\to\KO_\Q)$, and let
$\widehat\eta\in\pi_2C_{\KO}$ be the unique class whose Bockstein
is $\eta\in\pi_1\KO$. Then
\[
 \widehat\eta\,\eta^2=\frac12\alpha
 \quad\text{in }\pi_4C_{\KO}\cong(\Q/\Z)\alpha.
\]
\end{lemma}
\begin{proof}
The coefficient exact sequence gives
\[
 \pi_2C_{\KO}=\Z/2,\qquad \pi_3C_{\KO}=\Z/2,
 \qquad \pi_4C_{\KO}=(\Q/\Z)\alpha.
\]
For $C_{\KU}=\operatorname{cofib}(\KU\to\KU_\Q)$, odd homotopy
vanishes and $\pi_4C_{\KU}=(\Q/\Z)u^2$.
Smash Wood's cofiber sequence
\[
 \Sigma\KO\xrightarrow{\eta}\KO\xrightarrow{c}\KU
       \longrightarrow\Sigma^2\KO
\]
with the Moore spectrum $M(\Q/\Z)$
\cite[Section~3.2, equation~(3.12)]{TYYmod2}.
Its exact sequence includes
\[
 \pi_3C_{\KO}\xrightarrow{\eta}\pi_4C_{\KO}
 \xrightarrow{c}\pi_4C_{\KU}\longrightarrow
 \pi_2C_{\KO}\xrightarrow{\eta}\pi_3C_{\KO}
 \longrightarrow0.
\]
By~\eqref{eq:complexification}, $c$ is multiplication by two on
$\Q/\Z$, hence is surjective. Therefore the right-hand
$\eta$-map is an isomorphism, and the left-hand one has image
$\ker(c)=\{0,\alpha/2\}$. Applying these two maps to
$\widehat\eta$ gives the asserted class. Any boundary sign is
immaterial in these groups of order two.
\end{proof}

\begin{proposition}[Product formula in the bounded target]\label{prop:product-formula}
Let $k,k'\in\Z$, $a\in\pi_{8k+1}\TMF$, and
$a'\in\pi_{8k'+2}\TMF$, and set $K=k+k'$.
Write their primary images as
\[
 (\Phi_1)_*a=f\eta\beta^k,\qquad
 (\Phi_1)_*a'=f'\eta^2\beta^{k'},\qquad
 f,f'\in\F_2((q)).
\]
For any integral Laurent lifts $\widetilde f,\widetilde f'\in L$,
\begin{equation}\label{eq:product-formula}
 b_1(aa')=\left[\frac{\widetilde f\widetilde f'}2\right]
          \quad\text{in }\cV_{1,K}.
\end{equation}
\end{proposition}
\begin{proof}
For a ring spectrum $R$, write
$C_R=\operatorname{cofib}(R\to R_\Q)=R\wedge M(\Q/\Z)$.
It is an $R$-module. A ring map $\varphi:R\to A$ induces an
$R$-linear map $C_R\to C_A$, with $R$ acting on $C_A$ through
$\varphi$.

We first relate this coefficient construction to
Proposition~\ref{prop:relative}. If $x\in\pi_nR$ is torsion and
$\pi_nA=0$, choose a Bockstein lift
$\widehat x\in\pi_{n+1}C_R$. Its image in
\[
 \pi_{n+1}C_A=
 (\pi_{n+1}A\otimes\Q)/\im\pi_{n+1}A
\]
represents $b_\varphi(x)$, up to the fixed overall boundary sign,
after quotienting by the rational source image. Indeed, a rational
nullhomotopy of $x$ supplies $\widehat x$. Its image in $A$ differs
from an integral nullhomotopy by the rational representative used
in the cofiber definition. Changing the rational nullhomotopy adds
the image of $\pi_{n+1}R_\Q$, and changing the integral one adds
$\im\pi_{n+1}A$. More explicitly, if $G$ is the rational nullhomotopy of $x$ and
$H$ the integral nullhomotopy of $\varphi x$, both directed toward
zero, the cofiber representative is the loop
$(\varphi G)^{-1}H_\Q$. The coefficient-cofiber image is the inverse
loop $H_\Q^{-1}(\varphi G)$. They therefore differ by the single
overall sign. This is the coefficient-sequence diagram chase of \cite[Section~4.3, Lemma~4.4]{BN}; it uses exact sequences only,
and applies to periodic spectra in every degree.

Now take $R=\TMF$, $A=\KO((q))$, and $\varphi=\Phi_1$.
Rational odd homotopy of $R$ vanishes, and
\[
 \pi_{8k+2}R_\Q=M_{4k+1}^{!}(\mathrm{SL}_2(\Z);\Q)=0,
\]
since $-I$ acts by $-1$ in odd weight.
Thus $a$ has a unique Bockstein lift
$\widehat a\in\pi_{8k+2}C_R$. Both rational groups of $A$ in
degrees $8k+1$ and $8k+2$ also vanish, so its image is
$\widetilde f\beta^k\widehat\eta\in\pi_{8k+2}C_A$.
Because $a'$ has even degree, the module Bockstein sends
$\widehat a\,a'$ to $aa'$. Its image under $C_R\to C_A$ is
\[
 \widetilde f\widetilde f'\beta^K
       (\widehat\eta\,\eta^2)
   =\frac{\widetilde f\widetilde f'}2 e_K
\]
by Lemma~\ref{lem:KO-half}, transported through the constant-series
map $\KO\to A$. The product $aa'$ is torsion and its primary
image is zero, since $\pi_{8K+3}A=0$. The coefficient description
therefore proves~\eqref{eq:product-formula}; its class is killed
by two, so the overall boundary sign makes no difference.

Every representative displayed above has common denominator two.
Replacing either integral lift by that lift plus twice an element
of $L$ changes the representative by an element of $L$. Products
of Laurent series with finite lower bounds again have finite lower
bounds. These observations prove independence of lifts and cover
all integers $k,k'$, using the Bott unit $\beta$ in $A$.
\end{proof}

\begin{corollary}[Bunke--Naumann secondary invariants of products]\label{cor:Tachikawa}
For all $k,k'\in\Z$, $a\in\pi_{8k+1}\TMF$,
$a'\in\pi_{8k'+2}\TMF$, and $N>1$,
\[
 b_N^\partial\bigl(\iota_N(aa')\bigr)=0.
\]
\end{corollary}
\begin{proof}
The mathematical primary-image theorem of
Tachikawa--Yamashita--Yonekura
\cite[Section~4, Second statement, and Section~4.1]{TYYmod2}
supplies integral weakly holomorphic forms $H,H'$ of weights
$4k,4k'$ whose reductions are $f,f'$, respectively. For the
$8k+1$ case, apply the degree-$8k+2$ theorem to $\eta a$;
multiplication by $\eta$ sends $f\eta\beta^k$ to
$f\eta^2\beta^k$ without changing $f$.
The theorem is stated for periodic $\TMF$; its proof passes from
connective $\tmf$ by inverting the actual level-one periodicity
element $\Delta^{24}$, so negative degrees are included. Its mod-two generator images are
finite expressions in the integral forms $c_4$ and $\Delta$; their
coefficients can be lifted to $0$ or $1$. Inverting $\Delta^{24}$
preserves integral Laurent expansions. Thus the forms $H,H'$ have
integral coefficients, even though the source proof computes at
the prime two. No even-degree $\TMF$ lift of $H$ or $H'$ is required.

Set $h=HH'\in M_{4K}^{!}(\mathrm{SL}_2(\Z);\Z)$.
Proposition~\ref{prop:product-formula} gives $b_1(aa')=[h/2]$.
Corollary~\ref{cor:level-naturality} and
Theorem~\ref{thm:joint-kill} now give
\[
 b_N^\partial\bigl(\iota_N(aa')\bigr)
   =\rho_{N,K}^\partial([h/2])=0.
\]
This uses the mathematical primary-image result, without any
field-theoretic realization hypothesis or classification of $h$
as a single power of $\Delta$.
\end{proof}

\section{Torsion in the required stems of \texorpdfstring{$\TMF_0(3)$}{TMF0(3)}}

For the level-three calculation, we use the periodic spectrum denoted
$\TMF(\Gamma_0(3))$ by Mahowald--Rezk as our model of $\TMF_0(3)$.
Their Proposition~4.1 \cite{MR} gives the exact sequence
\begin{multline}\label{eq:MR}
0\longrightarrow
 \F_2[\Delta^{\pm2}]\{\nu,\nu^2,x,\eta x,\bar\kappa,x^2,\nu x^2\}
 \longrightarrow \pi_*\TMF_0(3)\\
 \longrightarrow
 bo_*[1/3,\Delta^{\pm1}]\{1,a_1a_3\}
 \oplus
 bsp_*[1/3,\Delta^{\pm1}]\{2a_3^2,2(a_1a_3)a_3^2\}
 \longrightarrow \Delta\F_2[\Delta^{\pm2}]\longrightarrow0.
\end{multline}
The classes appearing here have degrees
\[
 |\Delta|=24,\quad |a_1a_3|=8,\quad |a_3^2|=12,
 \quad |\eta|=1,\quad |\nu|=3,\quad |x|=17,
 \quad |\bar\kappa|=20.
\]
In this source exact sequence, $\Delta$ is the modular-form notation
used by Mahowald--Rezk. We denote their degree-$48$ spectral unit
$\Delta^2$ by $\delta$. The class $\nu\Delta^{2m}$ in their notation
is thus $\nu\delta^m$ here.

\begin{lemma}[Degree-three residue]\label{lem:MR-residue}
For every integer $m$,
\[
 \pi_{48m+3}\TMF_0(3)\cong\F_2\{\nu\delta^m\},
\]
and
\[
 \Tor\pi_n\TMF_0(3)=0
 \quad\text{for }n\equiv11,19,27,35,43\pmod{48}.
\]
\end{lemma}

\begin{proof}
The seven generators in the left term of~\eqref{eq:MR} have degrees
modulo eight
\[
 3,6,1,2,4,2,5,
\]
respectively. Because multiplication by $\Delta^{2}$ preserves these
residues, only the $\nu$ family occurs in residue three.

The coefficient groups of $bo$ are nonzero only in residues
$0,1,2,4$ modulo eight. The shifts $1$ and $a_1a_3$ differ by eight,
so the $bo$ term vanishes in residue three. Periodic symplectic
$K$-theory has nonzero coefficients in residues $0,4,5,6$; since
both displayed $bsp$ shifts are four modulo eight, that term is
nonzero only in residues $0,1,2,4$. The rightmost term, which is
supported in degrees $24+48\Z$, lies in residue zero. Exactness
of~\eqref{eq:MR} therefore identifies the middle group in residue
three with the indicated left summand, without any extension ambiguity.
\end{proof}

\begin{corollary}[The primary kernel]\label{cor:primary-kernel}
Every nonzero torsion class in a stem $8K+3$ lies in the kernel of the full-cusp primary map.
More precisely,
\[
 \ker(\Phi_3^\partial)_*\cap\Tor\pi_{8K+3}\TMF_0(3)
 =\begin{cases}
 \Z/2\{\nu\delta^m\},&K=6m,\\
 0,&K\not\equiv0\pmod6.
 \end{cases}
\]
\end{corollary}

\begin{proof}
Lemmas~\ref{lem:prime-cusps} and~\ref{lem:completed-coefficients}
give $\pi_{8K+3}\cK_3^\partial=0$, because real $K$-theory vanishes
in that residue. Every source class in the indicated degrees thus
belongs to the primary kernel, and Lemma~\ref{lem:MR-residue}
computes the source groups.
\end{proof}

\section{Residues and the exact level-three secondary value}\label{sec:residue}

A secondary value can vanish only if one rational global modular
form removes its two cusp coefficients modulo the real integral
lattices. In weight two, the residue theorem gives a linear
constraint on those two expansions. This constraint detects the
value of $\nu$ directly.

Write $S=\Z[1/3]$ and use
$\cK_3^\partial\simeq R_S(q)\times R_S(t)$ with the elliptic
Tate differential $d\log z$. In degree four we suppress the common
real basis $e_0=\alpha$. Write $\operatorname{CT}$ for Laurent
constant coefficient.

\begin{lemma}[The two-cusp residue relation]\label{lem:residue-relation}
For every $g\in M_2^!(\Gamma_0(3);\Q)$,
\begin{equation}\label{eq:residue-relation}
 \operatorname{CT}\Exp_\infty(g)
       +3\operatorname{CT}\Exp_0(g)=0.
\end{equation}
Consequently the formula
\begin{equation}\label{eq:residue-detector}
 \ell:\cV_{3,0}^\partial\longrightarrow\Q/S,
 \qquad [(u,v)]\longmapsto
   [\operatorname{CT}(u)+3\operatorname{CT}(v)]
\end{equation}
defines a homomorphism.
\end{lemma}
\begin{proof}
The differential $g(\tau)d\tau$ is invariant under $\Gamma_0(3)$.
It descends to a meromorphic differential on the compact modular
curve $X_0(3)$, with poles only at its two cusps. There are no
additional poles at elliptic points: in a local coordinate $s$
with effective stabilizer of order $e$, an invariant holomorphic
differential $a(s)ds$ has nonzero coefficients only in degrees
$e-1$ modulo $e$. It is therefore the pullback of a holomorphic
differential in the coarse coordinate $s^e$. The generic
involution acts trivially on a weight-two differential.

At infinity, $q=\exp(2\pi i\tau)$, so the residue is
$\operatorname{CT}\Exp_\infty(g)/(2\pi i)$.
At the other cusp, put $t=\exp(2\pi iz)$.
The differential convention in~\eqref{eq:tate-branches} gives
$\Exp_0(g)(t)=(g|S_0)(3z)$, just as in the first step of the
proof of Lemma~\ref{lem:prime-trace}. Setting $\tau=-1/(3z)$
therefore gives
\[
 g(\tau)d\tau=3\Exp_0(g)(t)\,dz
             =\frac3{2\pi i}\Exp_0(g)(t)\frac{dt}{t}.
\]
The residue theorem proves~\eqref{eq:residue-relation}, including
when $g$ has finite poles at the cusps. No holomorphic-only
$U_3$ identity is applied to $g$ here.

The numerator in~\eqref{eq:joint-value} has rational constant
coefficients. Its integral indeterminacy is
$S((q))\times S((t))$, which the displayed functional sends into
$S$. By Lemma~\ref{lem:rational-source}, every rational source
class is the pair
$\tfrac12(\Exp_\infty g,\Exp_0 g)$ for a single such $g$.
Equation~\eqref{eq:residue-relation} annihilates this pair exactly,
including its factor $1/2$. Hence the functional descends to the
actual bounded quotient.
\end{proof}

\begin{lemma}[The base value from the sphere unit]\label{lem:e-nonzero}
With the usual choice of sign for the stable class $\nu$,
\begin{equation}\label{eq:nu-base}
 b_3^\partial(\nu)
   =\pm\left[\left(\frac1{24},\frac1{24}\right)\right]
   =\pm\left[\left(\frac{E_2(q)}{24},
                         \frac{E_2(t^3)}{24}\right)\right].
\end{equation}
This value is nonzero and has exact order two.
\end{lemma}
\begin{proof}
For the sphere unit $\mathbb S\to\KO$, the relative invariant
in degree three is the real $e$-invariant with value
$\pm\alpha/24$ modulo $\Z\alpha$
\cite[Section~7, equation~(7.3), Theorem~7.16 and Example~7.17]{AdamsJIV}.
The normalization matters: Adams's $e'_R$ uses a bottom sphere
of dimension zero modulo eight. On the two-cell cone of a stable
representative of $\nu$, a real K-theory lift of the bottom
unit has rational top coefficient measured in the generator whose
complexification is twice the complex generator. To see the cofiber comparison explicitly, the extension of the
bottom unit over this two-cell cone induces a map from its top
sphere into the cofiber of $\mathbb S\to\KO$. The boundary is
$\nu$ up to the cofiber-rotation sign. Rationally the bottom cell
is killed, leaving precisely this top coefficient of $\alpha$;
changing the extension adds an integral multiple of $\alpha$.
Thus the cited value uses the real lattice of
\eqref{eq:complexification}, with no extra factor of two.

The constant-series inclusions give a homotopy-commutative square
\[
\begin{CD}
 \mathbb S @>>> \KO\\
 @VVV @VVV\\
 E_3 @>{\Phi_3^\partial}>> R_S(q)\times R_S(t).
\end{CD}
\]
Both composites from $\mathbb S$ are the unit. Naturality gives
the first equality in~\eqref{eq:nu-base} directly in the full
cusp target. The second follows from
$E_2=1-24\sum_{n\ge1}\sigma_1(n)q^n$, since the difference is
an integral pair. The same unit square at level one gives
$b_1(\nu)=\pm[1/24]=\pm[E_2/24]$, consistent with
\cite[Section~3.7, equation~(31)]{BN}; no connective-to-periodic
representative comparison is required.

The residue homomorphism now gives
\[
 \ell\bigl(b_3^\partial(\nu)\bigr)
      =\pm\left[\frac{1+3}{24}\right]
      =\pm[1/6]\ne0\quad\text{in }\Q/\Z[1/3].
\]
Since $1/6$ has a denominator divisible by two, it is not in
$\Z[1/3]$. The source class has order two by
Lemma~\ref{lem:MR-residue}, so its nonzero additive image has
exactly that order.
\end{proof}

\begin{lemma}[The periodic unit at both cusps]\label{ass:representatives}
Let $\delta\in\pi_{48}E_3$ be the Mahowald--Rezk unit with
modular form $\Delta^2$. Its cusp image is
\begin{equation}\label{eq:delta-cusps}
 (\Phi_3^\partial)_*(\delta)
   =\bigl(\Delta(q)^2\beta^6,\Delta(t^3)^2\beta^6\bigr).
\end{equation}
For every $m\in\Z$, multiplication by $\delta^m$ and by its
cusp image gives isomorphisms on the source and on the corresponding
relative value groups.
\end{lemma}
\begin{proof}
Mahowald--Rezk fix $\Delta$ as the discriminant for the invariant
elliptic differential and identify $\Delta^2$ as the degree-$48$
periodicity generator \cite[Proposition~3.2, Proposition~4.1 and Proposition~5.2]{MR}.
Thus its filtration-zero modular form is $\Delta^2$, not merely
a nonzero multiple of it. Tate restriction has the stated expansions
on the same underlying curves, with $q=t^3$ on the second branch.
In degree $48$, complexification sends $\beta^6$ to $u^{24}$ and
is injective on each torsion-free real Laurent module. It therefore
identifies the actual integral cusp image as~\eqref{eq:delta-cusps}.

Both $\Delta(q)=q\prod_{n\ge1}(1-q^n)^{24}$ and
$\Delta(t^3)=t^3\prod_{n\ge1}(1-t^{3n})^{24}$ are integral
Laurent units. All their integer powers have integral coefficients
and finite lower Laurent bounds. On the spectral source, $\delta$
is already a unit. Its action takes integral target classes to
integral classes and rational source images to rational source
images; $\delta^{-m}$ gives the inverse action. Hence it induces
the asserted isomorphisms on the relative quotients in every period.
\end{proof}

\begin{theorem}[Exact value on the level-three generator]\label{thm:nu-value}
For every $m\in\Z$,
\[
 b_3^\partial(\nu\delta^m)
 =\rho_{3,6m}^\partial\!\left(
 \left[\frac{E_2\Delta^{2m}}{24}\right]\right)
\]
is nonzero of exact order two.
\end{theorem}
\begin{proof}
The unit $\delta^m$ has even degree. Lemmas~\ref{lem:unit}
and~\ref{ass:representatives} therefore give
\[
 b_3^\partial(\nu\delta^m)
   =(\Phi_3^\partial)_*(\delta^m)b_3^\partial(\nu).
\]
This action is an isomorphism of value groups, so
Lemma~\ref{lem:e-nonzero} gives nonvanishing and exact order two
for every positive or negative $m$. In the coordinate
$e_{6m}=\alpha\beta^{6m}$, a representative is
\[
 \frac1{24}\bigl(\Delta(q)^{2m},\Delta(t^3)^{2m}\bigr).
\]
Replacing the two factors $1$ by $E_2(q)$ and $E_2(t^3)$ changes
this pair by an integral Laurent pair, since $E_2-1$ is divisible
by $24$. The resulting pair is precisely the image under
$\rho_{3,6m}^\partial$ of the stated bounded level-one representative.
Its possible overall sign is immaterial because the class has
order two. No level-one $\Delta^2$ spectral unit is used.
\end{proof}

\begin{corollary}[Complete table]\label{cor:table}
All kernels in the table are restricted to the displayed torsion domain.
Let $r=K\bmod6$ and $n_r=8r+3$. In one $48$-stem period, the complete
torsion domain, primary kernel, secondary kernel, and image are
\[
\begin{array}{c|c|c|c|c|c}
 r&n_r&\Tor\pi_{n_r}\TMF_0(3)&\ker\Phi_3^\partial&\ker b_3^\partial&\im b_3^\partial\\
\hline
0&3&\Z/2\{\nu\}&\Z/2&0&\Z/2\\
1&11&0&0&0&0\\
2&19&0&0&0&0\\
3&27&0&0&0&0\\
4&35&0&0&0&0\\
5&43&0&0&0&0
\end{array}
\]
Multiplication by $\delta$ transports the table to every degree.
\end{corollary}

\section{Conclusing remarks}

The joint-annihilation theorem constructs one rational global
correction and an actual integral global remainder. Applied to the
product formula of Section~\ref{sec:products}, it proves vanishing
at every nontrivial level for the Bunke--Naumann secondary invariants
of the stated products. The product formula is credited to Tachikawa;
the argument here supplies its proof in the ordered, bounded target.

At level three, the residue detector and the real sphere-unit
calculation give exact order two for the periodic $\nu$ family.
The underlying degree calculation is already determined by
\cite[Proposition~4.1]{MR}. The intrinsic class $x$ and its powers
contribute no additional classes in stems $8K+3$; the relation
$x^3=\nu\delta$ returns to the periodic $\nu$ family. The
vanishing of a secondary value does not imply the vanishing of its
homotopy class, and the present construction supplies neither a
geometric index theorem nor a field-theoretic level-$N$ realization.

Joint annihilation uses only the forward comparison into the
specified completed target. No identification is made with a
quotient obtained by interchanging localization, completion,
and rationalization.

\section*{Acknowledgments}
We thank Yuji Tachikawa for helpful correspondence on the surrounding $\TMF$ program and generously providing his lecture notes. The initial exploration, source synthesis, and drafting of this project were substantially assisted by OpenAI's GPT-5.6 Sol and GPT-6 Astra under the authors' direction.\footnote{The authors initially used Anthropic's Claude Fable 5 to generate an 88-page set of notes for this multi-paper project; these notes were later corrected and revised by ChatGPT.} The authors have checked the arguments and take responsibility for the mathematical claims.

\appendix
\section{Completion and source conventions}\label{app:interfaces}

The integral global class $z_N$ and rational global source class $x_N$
satisfy~\eqref{eq:global-witness-equality}, transported
by~\eqref{eq:prime-to-composite}. The completion convention and
bounded numerator are fixed
in~\eqref{eq:ordered-completion}--\eqref{eq:bounded-numerator}.
The map from the inherited level-one target always uses the forward
comparison~\eqref{eq:old-new-comparison}.

The degree-four real coordinate has
$c(\alpha)=2u^2$, while the degree-$48$ periodic coordinate has
$c(\beta^6)=u^{24}$. These are different normalization steps:
the first fixes both the rational modular-form image and the real
secondary coefficient; the second determines the integral cusp
image of $\delta$. In particular, multiplication by a modular
form in a Laurent representative does not assert the existence of
a like-named unit in level-one $\TMF$.

For products, $L_\Q=L\otimes\Q$ is the numerator before passing
to the quotient. The proof produces a representative with common
denominator two and an integral modular lift $h=HH'$.
It does not require an inverse to a completion comparison or the
stronger classification of each nonzero value by one power of
$\Delta$.

\section{Degree bookkeeping for Lemma~\ref{lem:MR-residue}}

The generators of the left module in~\eqref{eq:MR} have degrees
\[
\begin{array}{c|ccccccc}
\text{generator}&\nu&\nu^2&x&\eta x&\bar\kappa&x^2&\nu x^2\\
\hline
\text{degree}&3&6&17&18&20&34&37\\
\text{degree mod }8&3&6&1&2&4&2&5.
\end{array}
\]
The $bo$ coefficient residues are $0,1,2,4$. For $bsp$, the
coefficient residues are $0,4,5,6$, and each of the two $bsp$
generators in~\eqref{eq:MR} shifts degree by four modulo eight.
The rightmost module is supported in degree $24$ modulo $48$.
Thus every other term of the exact sequence vanishes in residue
three, leaving no group extension to be determined.

\begin{remark}[A remaining transfer comparison]\label{prop:transfer-square}
The smooth level-forgetting map
$f:\mathcal M_0(3)\to\mathcal M_{\mathrm{ell}}[1/3]$
has the published $\TMF[1/3]$-module transfer
$f_!:E_3\to\TMF[1/3]$, with
\begin{equation}\label{eq:transfer-four}
 f_!f^*=4\,\mathrm{id}
\end{equation}
\cite[Proposition~3.5]{MR}. A separate question is whether a trace
on the ordered completed full-cusp targets can be identified with
this specified transfer in a homotopy-commutative square
\[
\begin{CD}
 E_3 @>{\Phi_3^\partial}>>\cK_3^\partial\\
 @V{f_!}VV @VV{\tr_\partial}V\\
 \TMF[1/3] @>{\Phi_{1;3}}>>R_{\Z[1/3]}(q),
\end{CD}
\]
where $\Phi_{1;3}$ uses the forward completion comparison.
Such a square would imply
\begin{equation}\label{eq:transfer-secondary}
 \tr_\partial b_3^\partial(y)=b_{1;3}(f_!y)
\end{equation}
by naturality, where $b_{1;3}=b_{\Phi_{1;3}}$.
This comparison is not proved here and is not used by any theorem
or application above. The residue calculation detects the
secondary value without identifying a spectral cusp transfer.
\end{remark}

\bibliographystyle{amsalpha}
\bibliography{references}

\end{document}